\documentclass[12pt]{amsart}

\usepackage{amsmath}
\usepackage{amsthm}
\usepackage{amssymb}
\usepackage{hyperref}
\usepackage[margin=1.0in]{geometry}

\numberwithin{equation}{section}

\newtheorem{prop}{Proposition}
\newtheorem{lemma}[prop]{Lemma}

\newtheorem{thm}[prop]{Theorem}
\newtheorem{cor}[prop]{Corollary}

\numberwithin{prop}{section}

\theoremstyle{definition}
\newtheorem{defn}[prop]{Definition}

\newtheorem{rmk}[prop]{Remark}

\newcommand{\del}{\partial}
\newcommand{\delb}{\bar{\partial}}
\newcommand{\brs}[1]{\left| #1 \right|}

\newcommand{\gG}{\Gamma}

\newcommand{\gD}{\Delta}

\newcommand{\gw}{\omega}
\newcommand{\ga}{\alpha}
\newcommand{\gb}{\beta}

\newcommand{\N}{\nabla}

\newcommand{\til}[1]{\widetilde{#1}}

\renewcommand{\bar}[1]{\overline{#1}}

\renewcommand{\i}{\sqrt{-1}}
\newcommand{\C}{\mathbb C}
\newcommand{\R}{\mathbb R}

\newcommand{\pd}{\partial}
\renewcommand{\kappa}{\varkappa}

\newcommand{\bz}{\bar{z}}

\newcommand{\bgb}{\bar{\beta}}

\newcommand{\HH}{\mathcal{H}}

\DeclareMathOperator{\Rc}{Rc}

\DeclareMathOperator{\tr}{tr}

\begin{document}

\title[Generalized K\"ahler-Ricci solitons on surfaces]{Classification of generalized K\"ahler-Ricci solitons on complex surfaces}

\begin{abstract} Using toric geometry we give an explicit construction of the compact steady solitons for pluriclosed flow first constructed in \cite{Streetssolitons}.  This construction also reveals that these solitons are generalized K\"ahler in two distinct ways, with vanishing and nonvanishing Poisson structure.  This gives the first examples of generalized K\"ahler structures with nonvanishing Poisson structure on non-standard Hopf surfaces, completing the existence question for such structures.  Moreover this gives a complete answer to the existence question for generalized K\"ahler-Ricci solitons on compact complex surfaces.  In the setting of generalized K\"ahler geometry with vanishing Poisson structure, we show that these solitons are unique.  We show that these solitons are global attractors for the generalized K\"ahler-Ricci flow among metrics with maximal symmetry.
\end{abstract}

\author{Jeffrey Streets}
\address{Rowland Hall\\
         University of California, Irvine}
\email{\href{mailto:jstreets@uci.edu}{jstreets@uci.edu}}

\author{Yury Ustinovskiy}
\address{Courant Institute of Mathematical Sciences, New York}
\email{\href{mailto:yu3@nyu.edu}{yu3@nyu.edu}}

\date{\today}

\maketitle

\section{Introduction}

The pluriclosed flow is an extension of K\"ahler-Ricci flow to the setting of complex, non-K\"ahler geometry \cite{PCF}.  Crucial to understanding the singularity formation is to first understand the self-similar solutions, or solitons.  For a pluriclosed structure $(g, J)$, the steady soliton equations can be expressed as
\begin{gather} \label{f:solitonricci}
\begin{split}
\Rc - \tfrac{1}{4} H^2 + \N^2 f = 0,\\
\tfrac{1}{2} d^* H + i_{\frac{1}{2} \N f} H = 0,
\end{split}
\end{gather}
where $H = d^c \gw$ is a closed three-form, $H^2_{ij} = H_{ipq} H_j^{pq}$, and $f$ is a smooth function.  A surprising new feature of the pluriclosed flow is the existence of nontrivial compact steady solitons, constructed by the first author \cite{Streetssolitons}, specifically on all diagonal Hopf surfaces.  In this paper we give a new construction of these solitons from the point of view of toric geometry.  While the construction in \cite{Streetssolitons} reveals the Kaluza-Klein structure of these solitons, our construction here from the toric geometry point of view is much more explicit, and as such makes evident the connection to generalized K\"ahler geometry, which we will detail below.

\begin{thm} \label{t:mainthm1} On every diagonal Hopf surface $(M^4, I_{\ga\gb})$ there exists a pluriclosed steady soliton metric $g_{\ga\gb}$ which is invariant under the maximal torus of biholomorphisms, and is unique up to the $1$-parameter family of automorphisms generated by the associated soliton vector field.  Furthermore, it is
\begin{enumerate}
\item part of an odd generalized K\"ahler triple $(g_{\ga\gb}, I_{\ga\gb}, J)$, where $J$ is biholomorphic to $I_{\ga \bar{\gb}}$
\item part of an even generalized K\"ahler triple $(g_{\ga\gb}, I_{\ga\gb}, J)$, where $J$ is biholomorphic to~$I_{\ga \bar{\gb}}$.
\item A global attractor for the odd-type generalized K\"ahler-Ricci flow starting with invariant initial data.
\end{enumerate}
\end{thm}

\begin{rmk} It was not previously known if diagonal Hopf surfaces where $\brs{\ga} \neq \brs{\gb}$ admitted even-type generalized K\"ahler structure.  Our construction has exhibited such structures, along with a canonical metric of this type.
\end{rmk}

Combining Theorem \ref{t:mainthm1} with the rigidity results for solitons shown in \cite{Streetssolitons} and further structural results for generalized K\"ahler surfaces, we obtain a complete answer to the existence question for generalized K\"ahler steady solitons, and also the question of uniqueness among odd-type generalized K\"ahler structures.

\begin{cor} \label{c:cor1} Let $(M^4, g, I, J)$ be a compact steady generalized K\"ahler-Ricci soliton.  Then precisely one of the following holds:
\begin{enumerate}
\item $(M^4, I)$ is K\"ahler with $c_1(M) = 0$, $g$ is Calabi-Yau, and $I = \pm J$.
\item $(M^4, I)$ biholomorphic to a diagonal Hopf surface, specifically $I = I_{\ga\gb}$, and $J$ is biholomorphic to $I_{\ga \bar{\gb}}$.  In case the generalized K\"ahler structure has odd type, the metric $g$ is an invariant soliton as constructed in Theorem \ref{t:mainthm1}.
\end{enumerate}
\end{cor}

\textbf{Acknowledgements:} The authors thank Vestislav Apostolov and Max Pontecorvo for helpful conversations.
\section{Toric construction of solitons}
\subsection{Biholomorphisms of Hopf surfaces} \label{s:Hopfbackground}

Let $M$ be the diagonal Hopf surface with parameters $\ga,\beta \in \{z\in\C\ |\ 0<|z|<1\}$, i.e.
\begin{equation} \label{f:Hopfdef}
M = \C_{z_1,z_2}^2\backslash\{(0,0)\}/\Gamma,\quad
\Gamma\colon (z_1,z_2)\mapsto (\alpha z_1,\beta z_2).
\end{equation}
\noindent The action of the torus $(\C^\times)^2_{z_1,z_2}$ on $\C^2$ commutes with $\Gamma$ and descends to the action of $G=(\C^\times)^2_{z_1,z_2}/\Gamma\simeq T^3\times\R$ on $M$.  We use logarithmic coordinates $w_i = x_i+\sqrt{-1}y_i$ on the open dense set $(\C^\times)^2$, via
\begin{equation}
\exp\colon \C_{w_1,w_2}^2\to (\C^\times)_{z_1,z_2}^2, (w_1,w_2)\mapsto (\exp w_1,\exp w_2).
\end{equation}

A function or tensor descends to an $(S^1)^2$ invariant function or tensor on $(\C^\times)^2/\Gamma\subset M$ if and only if it is
invariant under the action of vector fields $\pd/\pd y_1,\ \pd/\pd y_2$, as well as the pull-back action of $\Gamma$:
	\[
	(w_1,w_2)\mapsto (w_1+\log \alpha,w_2+\log \beta).
	\]
In order to construct an explicit family of invariant metrics, we enforce these conditions even further, we will upgrade invariance under $\gG$ to invariance under a vector field which generates $\Gamma$.  In particular, let 
\begin{equation} \label{f:Zdef}
Z=\Re(\log\alpha)\pd/\pd x_1+\Re(\log\beta)\pd/\pd x_2.
\end{equation}
For simplicity we denote
\[
a=\Re(\log \alpha), \qquad b=\Re(\log\beta), \qquad Z=a\pd/\pd x_1+b\pd/\pd x_2.
\]
The time-one flow of the vector field $Z$ generates the action of $\gG$ modulo translation in $(y_1,y_2)$-plane. 
The function $\Re(-bw_1+aw_2)$ is invariant under $Z$, so making the linear change of coordinates
\[
u_1=\frac{b}{a} w_1 - w_2=\frac{b}{a}\log z_1-\log z_2,\quad u_2= w_2=\log z_2,
\]
invariance under the $(S^1)^2$ action and $Z$ is equivalent to the invariance with respect to the translations in $\Im(u_1)$, $\Im(u_2)$, $\Re(u_2)$. In what follows tensors and functions invariant under these 3 vector fields are referred to as invariant.

\subsection{Invariant metrics}

Up to multiplication by a positive constant, in the coordinates $(u_1,u_2)$, any Hermitian invariant metric is given by
\begin{equation} \label{f:invmetrics}
g_{(u_1,u_2)}=
\left[
\begin{matrix}
k & n+\sqrt{-1} m\\
n-\sqrt{-1}m & p
\end{matrix}
\right],
\end{equation}
where the real-valued functions $k(x),n(x),m(x),p(x)$ are evaluated at $x = 2\Re(u_1)=u_1+\bar{u_1}$.  

\begin{lemma}\label{lm:pluriclosed_constant}
	Let $g$ be a Hermitian metric on a compact complex surface $M$. Assume that $g$ admits a non-zero holomorphic Killing vector field $X\in\Gamma(M, T^{1,0}M)$. Then $g$ is pluriclosed if and only if $g(X,\bar X)$ is constant.
\begin{proof}
	Choose local coordinates $(z_1,z_2)$ on $M$ such that $X=\tfrac{\pd}{\pd z_1}$. The fundamental form $\omega_g$ is written as
	\[
	\omega_g=\sqrt{-1}(Fdz_1\wedge d\bar z_1+Gdz_1\wedge d\bar{z_2}+\bar G dz_2\wedge d\bar{z_1}+H dz_2\wedge d\bar{z_2})
	\]
	for some complex-valued functions $F,G,H$. Since $\tfrac{\pd}{\pd z_1}$ is a Killing vector field, $F,G,H$ depend only on $z_2$ and $\bar z_2$, therefore
	\[
	\pd\bar{\pd}\omega_g=-\sqrt{-1}F_{z_2\bar{z_2}}dz_1\wedge d{\bar z_1}\wedge dz_2\wedge d{\bar z_2}.
	\]
	Hence, $\omega_g$ is pluriclosed if and only if the function $g(X,\bar X)=F$ is pluriharmonic. On a compact surface the latter is equivalent to $F$ being a constant.
\end{proof}
\end{lemma}
Lemma~\ref{lm:pluriclosed_constant} implies that an invariant metric as in (\ref{f:invmetrics}) is pluriclosed if and only if $p \equiv const > 0$. Since we only deal with pluriclosed metrics, up to an overall scaling we can assume that $p \equiv 1$, which we will do henceforth.  
Switching back to $(z_1,z_2)$ coordinates, we have
\begin{equation}\label{f:invmetrics_z}
g_{(z_1,z_2)}=
\left[
\begin{matrix}
\cfrac{b^2}{a^2}\cdot\cfrac{k}{|z_1|^2} & \cfrac{b}{a}\cdot\cfrac{-k+n-\sqrt{-1}m}{\bar{z_1}z_2}\\
\cfrac{b}{a}\cdot\cfrac{-k+n+\sqrt{-1}m}{\bar{z_2}z_1} & \cfrac{k-2n+1}{|z_2|^2}
\end{matrix}
\right].
\end{equation}

Since we want to extend $g_{(z_1,z_2)}$ from the open part $(\C^\times)^2$ to the whole $\C^2\backslash\{(0,0)\}$, the functions $k(x), n(x), m(x)$ have to satisfy necessary asymptotic conditions.  Deriving the precise sufficient conditions is tedious and ultimately unnecessary in this generality, and instead we will show that the specific solitons we construct in this ansatz extend directly.

\begin{prop}\label{p:extension_necessary} If a pluriclosed metric as in (\ref{f:invmetrics_z}) extends to a smooth Hermitian metric on $\C^2\backslash\{(0,0)\}$ then $k(x),n(x),m(x)$ satisfy the following asymptotics (in each identity $C$ denotes a generic constant):
	\begin{enumerate}
		\item as $x\to +\infty$:
		\begin{enumerate}
			\item $k(x)-2n(x)+1=C\exp(-x)+
			\underline{o}(\exp(-x))$, where $C>0$,
			\item $\log(k(x)-2n(x)+1)'=-1+\overline{O}(\exp(-\frac{x}{2}))$,
			\item $k(x)-n(x)=\overline{O}(\exp(-\frac{x}{2}))$,
			\item $m(x)=\overline{O}(\exp(-\frac{x}{2}))$.
		\end{enumerate}
		\item as $x\to -\infty$
		\begin{enumerate}
			\item $k(x)=C\exp(\frac{a}{b}x)+
			\underline{o}(\exp(\frac{a}{b}x))$, where $C>0$,
			\item $\log(k(x))'=\frac{a}{b}+\overline{O}(\exp(\frac{a}{2b}x))$,
			\item $k(x)-n(x)=\overline{O}(\exp(\frac{a}{2b}x))$,
			\item $m(x)=\overline{O}(\exp(\frac{a}{2b}x))$.
		\end{enumerate}
	\end{enumerate}
	
	\begin{proof} We give the elementary proof of (2a) and (2b), the remaining cases being straightforward exercises.  The metric $g_{(z_1,z_2)}$ defined by formula \eqref{f:invmetrics_z} extends to a smooth metric on $\C^2\backslash\{(0,0)\}$ if and only if each of the functions
		\[
		\frac{k}{|z_1|^2}, \quad \frac{n-k+\sqrt{-1}m}{\bar{z_1}{z_2}}, \quad \frac{k-2n+1}{|z_2|^2},
		\]
		smoothly extends across
		\[
		\{z_1=0,z_2\neq 0\}\cup\{z_1\neq 0,z_2=0\}\subset \C^2\backslash\{(0,0)\}.
		\]
		Denote $r_1=|z_1|$, $r_2=|z_2|$, and let us deduce asymptotic (2a) by examining extension of $k/|z_1|^2$ to $\{z_1=0,z_2\neq 0\}$. A necessary condition for a smooth extension is that for any $r_2\neq 0$ there exists a \emph{positive} finite limit:
		\begin{equation}
		\begin{split}
		\lim_{r_1\to 0} \frac{k\Bigl(\cfrac{b}{a}\log r_1^2-\log r_2^2\Bigr)}{r_1^2}.
		\end{split}
		\end{equation}
		This limit exists and positive for every $r_2\neq 0$ if and only if the limit
		\[
		\lim_{r_1\to 0}\frac{k\Bigl(\cfrac{b}{a}\log r_1^2-\log r_2^2\Bigr)}{\exp\Bigl(\cfrac{a}{b}\cdot(\cfrac{b}{a}\log r_1^2-\log r_2^2)\Bigr)}=
		\lim_{x\to -\infty}\frac{k(x)}{\exp(\frac{a}{b}x)}
		\]
		exists and positive. This is equivalent to (2a). 
		
		To prove (2b) we observe that 
		\[
		\frac{d}{dr_1}\log\left(k\Bigl(\cfrac{b}{a}\log r_1^2-\log r_2^2\Bigr)/r_1^2\right)=\frac{2}{r_1}\left(\frac{b}{a}(\log k)'_x-1\right)
		\]
		has to have finite limit for any $r_2\neq 0$ as $r_1\to 0$. This limit exists if and only if 
		\[
		\frac{(\log k)'_x-\frac{a}{b}}{\exp(\frac{a}{2b}x)}
		\]
		has finite limit as $x\to -\infty$. This proves (2b).
	\end{proof}
\end{prop}

\subsection{Construction of solitons}

\begin{defn} We say that a pluriclosed structure $(M^{2n}, g, I)$ is a \emph{steady soliton} if there exists a holomorphic vector field $X$ such that
\begin{align} \label{f:Bismutsoliton}
(\rho_B^{1,1})_I = L_X \gw,
\end{align}
where $\N^B = \N^{LC} + \tfrac{1}{2} g^{-1} d^c \gw$ is the associated Bismut connection, defining a representative of $c_1$ denoted $\rho_B$.  For more background on this condition and its equivalence to (\ref{f:solitonricci}) we refer the reader to \cite{Streetssolitons}.  We note that in (\cite{Streetssolitons} \S 3) an argument was provided in the case of complex surfaces why the vector field $X$ must a priori be holomorphic, but in fact this follows directly from examining the equation (\ref{f:Bismutsoliton}) and its exterior derivative, in any dimension.
\end{defn}

\begin{lemma} \label{l:Bismutcurvature} Given an invariant pluriclosed metric as in (\ref{f:invmetrics}), one has
\begin{equation}\label{eq:ricci}
\begin{split}
(\rho_B^{1,1})_{(u_1,u_2)} =-&\ \left[
\begin{matrix}
\left(\frac{k'}{V} \right)' & \left( \frac{n' + \i m'}{V} \right)'\\
\left( \frac{n' - \i m'}{V} \right)' & 0
\end{matrix}
\right],
\end{split}
\end{equation}
where
\[
V=\det g_{(u_1,u_2)}=k-n^2-m^2.
\]
\begin{proof}
	We will use the following identity:
	\[
	\rho^{1,1}_B=\rho_C-\pd\pd^*\omega-\bar{\pd}\bar{\pd^*}\omega=
	-\pd\left(\frac{\sqrt{-1}}{2}\bar{\pd}\log V + \pd^*\omega\right)-\bar{\pd}\left(-\frac{\sqrt{-1}}{2}\pd\log V + \bar{\pd^*}\omega\right),
	\]
	where
	$\rho_C=-\sqrt{-1}\pd\bar{\pd}\log V$ is the Chern-Ricci form. For $\pd^*\omega$ we have a general formula
	\[
	(\pd^*\omega)_{\bar j}=\sqrt{-1}g^{m\bar n}(\pd_{\bar n}g_{m\bar j}-\pd_{\bar j}g_{m\bar n}).
	\]
	Now, using that $V$ and all the entries of $g_{(u_1,u_2)}$, $g^{-1}_{(u_1,u_2)}=\frac{1}{V}\left[
	\begin{matrix}
	1 & -n-\sqrt{1}m\\
	-n+\sqrt{-1}m & k
	\end{matrix}
	\right]$ depend only on $u_1+\bar{u_1}$, we find
	\begin{equation}\label{eq:aux:ricci_1}
	\begin{split}
	\frac{\sqrt{-1}}{2}\bar{\pd}\log V + \pd^*\omega
	&=
	\sqrt{-1}
	\left(
	\frac{V'}{2V}d{\bar u_1}+\frac{n'+\sqrt{-1}m'}{V}d{\bar u_2}+\frac{(n'+\sqrt{-1}m')(n-\sqrt{-1}m)}{V}d\bar{u_1}
	\right)\\
	&=\sqrt{-1}
	\left(
	\frac{k'}{2V} d\bar{u_1}+\frac{n'+\sqrt{-1}m'}{V}d\bar{u_2}
	\right)
	+
	\frac{n'm-m'n}{V}d{\bar u_1}.
	\end{split}
	\end{equation}
	Similarly,
	\begin{equation}\label{eq:aux:ricci_2}
	-\frac{\sqrt{-1}}{2}\pd\log V + \bar{\pd^*}\omega=
	-\sqrt{-1}
	\left(
	\frac{k'}{2V} d u_1+\frac{n'-\sqrt{-1}m'}{V}d u_2
	\right)
	+
	\frac{n'm-m'n}{V}du_1.
	\end{equation}
	Since all the coefficients of the forms in~\eqref{eq:aux:ricci_1} and~\eqref{eq:aux:ricci_2} again depend only on $u_1+\bar{u_1}$, we directly verify that
	\[
	\rho_B^{1,1}=-\pd\left(\frac{\sqrt{-1}}{2}\bar{\pd}\log V + \pd^*\omega\right)-\bar{\pd}\left(-\frac{\sqrt{-1}}{2}\pd\log V + \bar{\pd^*}\omega\right)
	\]
	is given by~\eqref{eq:ricci}.
\end{proof}
\end{lemma}

\begin{prop} \label{p:existence} Let $(M^4, I_{\ga\gb})$ be a diagonal Hopf surface.  There exists a pluriclosed steady soliton $\gw_{\ga\gb}$ on $(M^4, I_{\ga\gb})$ given by
\begin{align} \label{f:soliton}
\gw_{\ga\gb} =&\ \i \left( \frac{b^2}{a^2} \frac{k}{\brs{z_1}^2} dz_1 \wedge d \bz_1 + \frac{1-k}{\brs{z_2}^2} dz_2 \wedge d \bz_2 \right),
\end{align}
where $k = k(x)$, $x=\frac{b}{a} \log \brs{z_1}^2 - \log \brs{z_2}^2$ is a solution of the ODE
\begin{align*}
k' = k(1-k)\left( \left(1- \frac{a}{b} \right) k + \frac{a}{b} \right).
\end{align*}
\begin{proof}
Let $(u_1,u_2)$ be the coordinates which we have constructed above. The automorphism group of a generic Hopf surface has real dimension 4 with infinitesimal action generated by $\Re(u_1)$, $\Im(u_1)$, $\Re(u_2)$, $\Im(u_2)$~\cite{NambaHopfAut}. Since our metrics are invariant under translations in $\Im(u_1)$, $\Im(u_2)$, $\Re(u_2)$, it is natural to look for a soliton with the drift vector field of the form $\mu Y$, where $Y=\pd_{\Re(u_1)}$, $\mu\in\R$:
\[
\rho_B^{1,1} = \mu  L_{Y}g_{(u_1,u_2)}.
\]
By an elementary computation of $L_Y g_{(u_1,u_2)}$ together with Lemma \ref{l:Bismutcurvature} we reduce soliton equation to the system of equations
\begin{equation}\label{eq:pluriclosed_soliton}
\left\{
\begin{split}
\left(\frac{k'}{V}\right)'&=\mu k'\\
\left(\frac{m'}{V}\right)'&=\mu m' \\
\left(\frac{n'}{V}\right)'&=\mu n'
\end{split}
\right.
\end{equation}
Integrating each equation once, we obtain
\begin{equation}
\left\{
\begin{split}
V&=\frac{k'}{\mu k+c_k}\\
V&=\frac{m'}{\mu m+c_m}\\
V&=\frac{n'}{\mu n+c_n}
\end{split}
\right.
\end{equation}
Comparing the right hand sides of these equations, we conclude that $m$ and $n$ are affine functions of $k$, i.e., $m=a_mk+b_m$, and $n=a_nk+b_n$. Since $V=k-n^2-m^2$ is positive, and using the asymptotics of $k$ as $x\to-\infty$ from Proposition \ref{p:extension_necessary}, we conclude that $b_n=b_m=0$, i.e., the functions $k,m,n$ are proportional to each other.  In this case, the off-diagonal elements of $g_{(z_1,z_2)}$ have the correct asymptotic as $x\to-\infty$ iff they vanish, i.e., we have
\[
k=n,\quad m=0.
\]
Thus we are reduced to solving a single ODE for $k$:
\begin{equation}\label{eq:ODE}
k'=k(1-k)(\mu k+c_k).
\end{equation}
To identify constants $\mu$ and $c_k$, we further analyze behavior of $k(x)$ as $x\to\pm\infty$. By Proposition~\ref{p:extension_necessary} (2b) we know that
\[
\frac{k'}{k}-\frac{a}{b}\to 0 \quad \mbox { as } x\to-\infty,
\]
while at the same time 
\[
(1-k)(\mu k+c_k)\to c_k \quad \mbox { as } x\to-\infty,
\]
so $c_k=\frac{a}{b}$. Similarly considering behavior at $+\infty$ we find $\mu+c_k=1$.
\begin{equation}
c_k=a/b,\quad \mu=1-a/b.
\end{equation}

If $a=b$, then $M$ is a Hopf surface with $|\alpha|=|\beta|$, $\mu=0$ (i.e., $g$ is a stationary metric). In this case $k(x)=1/({1+Ce^{-x}})$, $C>0$, and $g$ is equivalent to the round metric $g_{\mathrm{Eucl}}/|z|^2$ via an automorphism of $(M,I)$.

From now on assume $a\neq b$. By separation of variables, the ODE~\eqref{eq:ODE} is thus equivalent~to
\begin{equation*}
x+const=\frac{b}{a-b}\int \frac{1}{k(k-1)(k-\frac{a}{a-b})}\,dk.
\end{equation*}
Computing the integral, we conclude that
\begin{equation}\label{eq:final}
x+const=-\log(1-k)+\log\Bigl(\frac{a}{a-b}-k\Bigr)\frac{a-b}{a}+\frac{b}{a}\log k.
\end{equation}
The function $x(k)$ monotonically increases on $(0,1)$ from $-\infty$ to $+\infty$, so~\eqref{eq:final} yields a well-defined function
$k(x)\colon\R\to(0,1)$. The corresponding metric in the original $z$ coordinates is
\begin{equation} \label{f:solitonzcoords}
g_{(z_1,z_2)}=
\left[
\begin{matrix}
\cfrac{b^2}{a^2}\cdot\cfrac{k}{|z_1|^2} & 0\\
0 & \cfrac{1-k}{|z_2|^2}
\end{matrix}
\right].
\end{equation}

It remains to check that the metric defined by $k(x)$ on $(\C^*)^2_{(z_1,z_2)}$ extends to the whole $\C^2\backslash\{(0,0)\}$. Indeed, in the neighbourhood of a point $(z_1,0)$ with $|z_1|\neq 0$ we have
\[
r_2^2=r_1^{\frac{2b}{a}}(1-k)\left(\frac{a}{a-b}-k\right)^{-\frac{a-b}{a}}k^{-\frac{b}{a}},\quad r_1=|z_1|, r_2=|z_2|,
\]
and $r_2$ is a smooth function of $(r_1,k)$ in the neighbourhood of $(r_1,1)$. Hence the implicitly defined function $k$ extends to a smooth function of $r_1$, $r_2$ in the neighbourhood of $(r_1,0)$. This ensures that $(g_{(z_1,z_2)})_{2\bar 2}$ extends across $z_2=0$. Similarly, $(g_{(z_1,z_2)})_{1\bar 1}$ extends across $z_1=0$.  Finally, we note that the choice of the constant of integration in~\eqref{eq:final} corresponds to the choice of the time slice of the corresponding pluriclosed soliton.
\end{proof}
\end{prop}

\section{Compatible generalized K\"ahler structures} \label{s:compGK}

We recall that a generalized K\"ahler manifold $(M^{2n}, g, I, J)$ consists of a Riemannian metric compatible with two integrable complex structures $I$ and $J$, such that furthermore
\begin{align*}
d^c_I \gw_I = H = - d^c_J \gw_J, \qquad d H = 0.
\end{align*}
For more background on these structures consult \cite{GualtieriThesis}.  The pluriclosed flow preserves generalized K\"ahler geometry \cite{GKRF}, and thus it is natural to expect that the solitons we have constructed should also be global attractors for the generalized K\"ahler-Ricci flow, and hence should be generalized K\"ahler structures.  We confirm this in this section.

\begin{defn} We say that a generalized K\"ahler structure $(M^{2n}, g, I, J)$ is a \emph{steady generalized K\"ahler-Ricci soliton} if there exists a vector field $X$, holomorphic with respect to both $I$ and $J$, such that
\begin{align*}
(\rho_B^{1,1})_I = L_X \gw_I.
\end{align*}
\end{defn}

\subsection{Odd-type generalized K\"ahler structure} \label{ss:oddGK}

In four dimensions, odd-type generalized K\"ahler structure (cf \cite{GualtieriThesis}) consists of a generalized K\"ahler triple triple $(g, I, J)$ such that $[I, J] = 0$, and $I \neq \pm J$, and corresponds precisely to the case of generalized K\"ahler for which $I$ and $J$ induce opposite orientations.  Such structures in general yield holomorphic splittings of the tangent bundle according to the eigenspaces of $\Pi = IJ$ (\cite{ApostolovGualtieri} Theorem 4).  Furthermore, Apostolov-Gualtieri classified such structures on compact complex surfaces (\cite{ApostolovGualtieri} Theorem 1), and among Hopf surfaces these exist precisely on a certain subclass of diagonal Hopf surfaces, and all are finitely covered by a primary diagonal Hopf surface.  For these examples the pair of complex structures are determined by the standard complex structure on $\mathbb C^2$ together with the complex structure where the orientation of the $z_2$-plane is reversed.  These both which descend to every quotient diagonal Hopf surface, and we refer to the them as $I_{\ga\gb}, J_{\ga\gb}$.  Incidentally, using the map $(z_1, z_2) \to (z_1, \bar{z}_2)$ it follows that $J_{\ga\gb}$ is biholomorphic to $I_{\ga \bgb}$.  As explained in (\cite{ApostolovGualtieri} pg. 18), any metric compatible with commuting complex structures $I$ and $J$ which is pluriclosed with respect to $I$ is automatically generalized K\"ahler.  By inspecting formula (\ref{f:solitonzcoords}), the solitons are manifestly compatible with both complex structures, and by construction are also pluriclosed, hence $(g_{\ga \gb}, I_{\ga\gb}, I_{\ga\bgb})$ forms a generalized K\"ahler triple, as claimed.

\subsection{Even-type generalized K\"ahler structure}

In four dimensions, even type generalized K\"ahler structure is equivalent to asking that $I$ and $J$ induce the same orientation.  In this setting a key role is played by the set
\begin{align*}
T = \{p \in M\ |\ I(p) = \pm J(p) \}.
\end{align*}
As shown in (\cite{AGG, PontecorvoCS}), $T$ is the support of a common effective divisor in each complex surface.  Furthermore, if $M$ has odd first Betti number, then $T$ is disconnected, and the minimal model of $M$ must be either a diagonal Hopf surface, or a parabolic or hyperbolic Inoue surface (\cite{FujikiPontecorvoNKT} Proposition 4.7, cf. \cite{AGG}).  Observe that it follows from this fact that even-type generalized K\"ahler structure on primary diagonal Hopf surfaces must have the two complex structures biholomorphic to $I_{\ga\gb}$ and $I_{\ga \bgb}$ for some $\ga, \gb$.  This is because $T$ must consist of the two elliptic curves, and the two complex structures must induce the same orientation on one curve (the component of $T$ where $I = J$) and the opposite orientation on the other (the component of $T$ where $I = -J$) (cf. \cite{AGG} Proposition 4).

Constructions of generalized K\"ahler structures exist on some parabolic Inoue surfaces (\cite{LeBrun}, \cite{FujikiPontASD}), and in fact on all hyperbolic Inoue surfaces (\cite{FujikiPontASD}).  Despite these intricate constructions, in the relatively simpler case of Hopf surfaces, the only known examples are on the standard Hopf surfaces satisfying $\ga = \gb$ (\cite{GualtieriGKG} Example 2.23).  Below we show that our solitons are in fact compatible with an even-type generalized K\"ahler structure, thus completing the existence question for generalized K\"ahler structure on Class $\mbox{VII}_0$ surfaces.  The question of the topology of the space of such structures on Hopf surfaces remains open.

\begin{prop} \label{p:eventype} Given $(M^4, I_{\ga\gb})$ a diagonal Hopf surface, any metric of the form
	\begin{equation} \label{f:aux:solitonzcoords}
	g_{(z_1,z_2)}=
	\left[
	\begin{matrix}
	\cfrac{b^2}{a^2}\cdot\cfrac{k}{|z_1|^2} & 0\\
	0 & \cfrac{1-k}{|z_2|^2}
	\end{matrix}
	\right]
	\end{equation}
is a part of an even-type generalized K\"ahler triple $(g, I_{\ga\gb}, J)$, where $J$ is biholomorphic to $I_{\ga \bgb}$. In particular, if $g_{\ga\gb}$ denotes the metric associated to the K\"ahler form $\gw_{\ga\gb}$ from Proposition~\ref{p:existence} then $(g_{\ga\gb}, I_{\ga\gb}, J)$ is a generalized K\"ahler-Ricci soliton.

\begin{proof} Let us first exhibit the second complex structure $J$ with which $g$ is compatible.  We use the logarithmic coordinates $w$ defined in \S \ref{s:Hopfbackground} to define the complex structure via a holomorphic volume form defined on an open dense subset.  For instance, the complex structure $I_{\ga\gb}$ is the complex structure for which the form
\begin{align*}
\Omega_- = dw_1 \wedge dw_2
\end{align*}
is of type $(2,0)$.  Likewise, let $J$ denote the unique complex structure for which
\begin{align} \label{f:Omegaplus}
\Omega_+=(dw_1- \frac{a}{b} d\bar{w_2})\wedge\Bigl(\frac{b}{a}k d\bar{w_1}+(1-k)dw_2\Bigr)=:\phi_1\wedge\phi_2
\end{align}
is of type $(2,0)$.  This is integrable since the sub-bundle of $\Lambda^{1}_\C(M)$ spanned by $\phi_1$ and $\phi_2$ is Frobenius-involutive:
\[
d\phi_i\wedge\phi_1\wedge\phi_2=0,\quad 1\le i\le 2.
\]
Furthermore, $g_{\ga\gb}$ is Hermitian with respect to $J$ since direct computation shows that
\[
g_{\ga\gb}(\phi_i,\phi_j)=0,\quad 1\le i,j\le 2.
\]
Further direct computation shows that $\Omega_{\pm}$ have the same real parts, and that
\begin{align*}
- \pi^{1,1}_{I_{\ga\gb}} \Im \Omega_+ = \i \left(\frac{b}{a} k dw_1 \wedge d \bar{w}_1 + \frac{a}{b} (1-k)d w_2 \wedge d \bar{w}_2 \right).
\end{align*}
Comparing against (\ref{f:invmetrics_z}), we see that the right hand side is a constant multiple of the $\gw_{\ga\gb}$, expressed in $w$ coordinates, and is in particular positive definite.  Thus, according to a well-known description of generalized K\"ahler metrics due to Joyce (cf. for instance \cite{ASNDGKCY} Lemma 2.14), the triple $(g, I_{\ga\gb}, J)$ defines a generalized K\"ahler structure on an open dense subset.  By construction the metric extends smoothly to the two elliptic curves.  An inspection of (\ref{f:Omegaplus}) shows that $J$ also extends smoothly to the two elliptic curves, and that the induced orientation on the curve $z_2 = 0$ agrees with that induced by $I_{\ga\gb}$, whereas the induced orientation on the curve $z_1 = 0$ is opposite that induced by $I_{\ga\gb}$.  As the biholomorphism type of a complex structure on a Hopf surface is determined by the periods of the elliptic curves, it follows that $J$ is biholomorphic to $I_{\ga \bar{\gb}}$.
\end{proof}
\end{prop}

\section{Pluriclosed flow of invariant metrics}

In this section we prove that generalized K\"ahler-Ricci flow with invariant initial data converges to the pluriclosed soliton.  We then collect the proof of Theorem \ref{t:mainthm1} from this and the results above.

\begin{prop} \label{p:flowconv} Let $(M^4, I_{\ga\gb})$ be a diagonal Hopf surface, and suppose $g$ is an invariant pluriclosed metric of the form
		\begin{equation}\label{eq:aux:gz}
		g_{(z_1,z_2)}=
		\left[
		\begin{matrix}
		\cfrac{b^2}{a^2}\cdot\cfrac{k}{|z_1|^2} & 0\\
		0 & \cfrac{1-k}{|z_2|^2}
		\end{matrix}
		\right].
		\end{equation}
	 The solution $\gw_t$ to pluriclosed flow with this initial data remains invariant, exists for $[0,\infty)$, and converges in $C^{\infty}$ to $\gw_{\ga\gb}$, the unique invariant soliton of Proposition \ref{p:existence}.
\begin{proof}
	Let $k_0(x)\colon \R\to (0,1)$ by any smooth function such that~\eqref{eq:aux:gz} defines a smooth metric on Hopf surface $M$. Denote by $\kappa(x)$ the function corresponding to the soliton metric of Proposition~\ref{p:existence} characterized by ODE
	\[
	\kappa'_x=\kappa(1-\kappa)
	\left(
		\frac{b-a}{b}\kappa+\frac{a}{b}
	\right).
	\]
	Let $g(t)$ be the solution to the pluriclosed flow on the maximal time interval $[0,T_{\max})$ and denote by $k=k(x,t)$ the corresponding solution to:
	\begin{equation} \label{f:kflow}
	\begin{cases}
	k'_t= \left(\frac{k'_x}{k(1-k)}\right)'_x\\
	k(x,0)=k_0(x)
	\end{cases}
	\end{equation}
	on $\R_x\times[0,T_{\max})_t$.  As the pluriclosed flow preserves invariance under Killing fields it follows that the metric $g(t)$ is described by the ansatz (\ref{eq:aux:gz}), where $k$ is determined by (\ref{f:kflow}).
	\begin{lemma} \label{l:kestimate}
		There exists a constant $C>0$ such that for any $t\in[0,T_{\max})$
		\[
		\kappa(x-C)<k(x-\mu t,t)<\kappa(x+C),
		\]
		where $\mu=\frac{b-a}{b}$.
	\begin{proof}
		The asymptotics of $\kappa$ and $k(x,0)$ as $x\to \pm \infty$ imply that for some constant $C>0$
		\[
		\kappa(x-C)<k(x,0)<\kappa(x+C).
		\]
		We will prove that the same constant $C$ works for any $t>0$.

		Let 
		\[
		s(x,t):=\kappa^{-1}(k(x-\mu t,t)).
		\]
		Then the statement of the lemma is equivalent to the bound $|s(x,t)-x|<C$ on $\R_x\times [0,T_{\max})$. Let us derive the evolution equation for $s(x,t)$. We directly compute using the differential equations satisfied by $k(x)$ and $\kappa(s)$:
		\begin{equation}
		\begin{split}
		s'_x=&\ \frac{1}{\kappa'_s}k'_x\\
		s'_t=&\ \frac{1}{\kappa'_s}(k'_t-\mu k'_x)=
		\frac{1}{\kappa'_s}
		\left(
			\Bigl(\frac{k'_x}{k(1-k)}\Bigr)'_x-\mu k'_x
		\right)\\
		=&\ \frac{1}{\kappa'_s}
		\left(
		(s'_x (\mu k+a/b))'_x-\mu k'_x
		\right)\\
		=&\ \frac{s''_x}{k(1-k)}+\mu (s'_x)^2-\mu s'_x.
		\end{split}
		\end{equation}
		
		Therefore the function $\Phi(x,t):=s(x,t)-x$ satisfies
		\begin{equation}\label{eq:Phi}
		\Phi'_t=\frac{\Phi''_x}{k(1-k)}+\mu \Phi'_x (\Phi'_x+1).
		\end{equation}
		To obtain an upper bound on $|\Phi(x,t)|$ it remains to prove that equation~\eqref{eq:Phi} for a function on a non-compact domain satisfies a parabolic maximum principle. Equation~\eqref{eq:Phi} is equivalent to an equation
		\[
		\Phi'_t=\Delta^C \Phi+\mu \nabla_Y\Phi(\nabla_Y\Phi-1)
		\]
		for $\Phi(x,t)$, $x=\frac{b}{a}\log |z_1|^2-\log |z_2|^2$ on $(\C^*)^2_{(z_1,z_2)}\times [0,T_{\max})$, where $\nabla$ is the Chern connection, $\Delta^C=g^{i\bar j}\nabla_i\nabla_{\bar j}$ is the Chern Laplacian, and $\mu Y$ is the soliton vector field.  Furthermore, smooth dependence of $g(t)$ on $t$ and asymptotics of $k(x,t)$ and $\kappa(x)$ at $\pm\infty$ ensure that $\Phi(x,t)$, $x=\frac{b}{a}\log |z_1|^2-\log |z_2|^2$ on $(\C^*)_{(z_1,z_2)}^2\times [0,T_{\max})$ extends to a $C^2$ function on $(\C^2\backslash\{(0,0)\})\times [0,T_{\max})$.  Thus $\Phi(x,t)$, $x=\frac{b}{a}\log |z_1|^2-\log |z_2|^2$ solves a parabolic PDE on $(\C^2\backslash\{(0,0)\})$ and being invariant descends to a PDE on the Hopf surface $M=(\C^2\backslash\{(0,0)\})/\Gamma$. Hence we can apply standard parabolic maximum principle and conclude that
		\[
		-C<\Phi(x,t)<C
		\]
		on $\R_x\times [0,T_{\max})$. This proves the lemma.
	\end{proof}
	\end{lemma}
	
	Lemma \ref{l:kestimate} implies that, after pulling back by the appropriate time-dependent translation in the $x$ variable, the function $k$ has uniform bounds relative to $\kappa$, which easily implies that the resulting metrics $g(t)$ are uniformly equivalent to the soliton metric.  Also one can obtain a torsion potential along the flow, which is given by
	\begin{align*}
	\gb = (k(x,t)-\kappa(x))du_1\wedge du_2 \in \Lambda^{2,0},\quad x=u_1+\bar{u_1}
	\end{align*}
	in $(u_1,u_2)$ coordinates, and extends smoothly to the whole Hopf surface. This torsion potential is characterized by the identity
	\begin{align*}
	\delb \gb = \del \gw_t - \del \gw_{\ga\gb}.
	\end{align*}
	From the uniform bound on $k$ it follows that $\gb$ is also uniformly bounded.  It then follows from the general regularity results for pluriclosed flow (cf. \cite{StreetsPCFBI} Theorem 1.7, Proposition 5.13) that the metric $g(t)$ has uniform $C^{\infty}$ bounds for all time.  Using the gradient flow property of pluriclosed flow (\cite{PCFReg}), a standard argument shows that any sequence contains a subsequence which converges to a steady soliton, which by uniqueness within the symmetry class must be $\gw_{\ga\gb}$.  The result follows.
\end{proof}
\end{prop}

\begin{proof}[Proof of Theorem \ref{t:mainthm1}] The claimed existence was shown in Proposition \ref{p:existence}.  The compatible odd and even type generalized K\"ahler structures are constructed in \S \ref{s:compGK}.  It follows from the discussion of odd-type generalized K\"ahler structures in \S \ref{ss:oddGK} that any invariant metric as described in (\ref{f:invmetrics_z}) is in fact of the form (\ref{eq:aux:gz}).  Thus the global existence and convergence of generalized K\"ahler-Ricci flow with this initial data is shown in Proposition \ref{p:flowconv}.
\end{proof}

\section{Uniqueness for odd-type generalized K\"ahler-Ricci solitons}

To address the uniqueness question in the odd-type generalized K\"ahler setting, we first introduce a certain cohomology space associated to generalized K\"ahler geometry in the case $[I,J] = 0$.  Let $(M^{2n}, g, I, J)$ be a generalized K\"ahler manifold with $[I,J] = 0$, and let $\Pi = IJ$ as above.  Given $\phi_I \in \Lambda^{1,1}_{I,\mathbb R}$, let 
\begin{align*}
\phi_J = - \phi_I(\Pi \cdot, \cdot) \in \Lambda^{1,1}_{J,\mathbb R}.
\end{align*}
We will say that $\phi_I$ is \emph{formally generalized K\"ahler} if
\begin{gather*}
\begin{split}
d^c_{I} \phi_I =&\ - d^c_{J} \phi_J,\\
d d^c_{I} \phi_I =&\ 0.
\end{split}
\end{gather*}
These are precisely the conditions satisfied by the K\"ahler form $\gw^I$ associated to a generalized K\"ahler metric, ignoring positivity.

As the $\pm 1$-eigenspace splitting for $\Pi$ is holomorphic, we obtain a fourfold splitting of the $d$ operator as
\begin{align*}
d = \del_+ + \delb_+ + \del_- + \delb_-.
\end{align*}
A key point is that for a smooth function $f$, the tensor
\begin{align} \label{f:squaredef}
\square f := \i \left( \del_+ \delb_+ - \del_- \delb_- \right) f
\end{align}
is formally generalized K\"ahler.  Using these facts we can define cohomology classes for formally generalized K\"ahler structures.

\begin{defn} \label{d:cone} Let $(M^{2n}, g, I, J)$ be a generalized K\"ahler manifold
such that $[I,J] = 0$.  Let
\begin{align*}
 \mathcal H := \frac{ \left\{ \phi_I \in \Lambda^{1,1}_{I, \mathbb R}\ |  \
\phi_I
\mbox{ is formally generalized K\"ahler} \right\}}{ \left\{ \i \left( \del_+ \delb_+ - \del_- \delb_- \right) f \ |\ f \in C^{\infty}(M,\R) \right\}}.
\end{align*}
Given $\phi$ formally generalized K\"ahler we denote its cohomology class in $\HH$ by $[\phi]$.
\end{defn}

In analogy with the Calabi-Yau theorem we can only expect uniqueness of our solitons to hold within a fixed cohomology class, and moreover with respect to a fixed choice of holomorphic vector field.  The next proposition establishes this fact in broad generality.

\begin{prop} \label{p:uniqueness} Let $(M^{2n}, g_{i}, I, J)$, $i=1,2$ be generalized K\"ahler-Ricci solitons with respect to a vector field $X$, and further assume $[I,J] = 0$ and $[\gw^I_1] = [\gw^I_2]$.  Then $\gw^I_1 = \gw^I_2$.
\begin{proof} Since $[\gw^I_1] = [\gw^I_2]$ there exists a smooth function $u \in C^{\infty}(M)$ such that
\begin{align*}
\gw^I_2 = \gw^I_1 + \square u.
\end{align*}
As both $\gw^I_i$ satisfy $\rho_B^{1,1}(\gw^I_i) = L_X \gw^I_i$, we directly obtain using the transgression formula for $\rho_B^{1,1}$ for generalized K\"ahler structures (cf. \cite{StreetsCGKFlow} Proposition 3.2),
\begin{align*}
0 =&\ \rho_B^{1,1}(\gw^I_2) - \rho_B^{1,1}(\gw^I_1) - L_X \left( \gw^I_2 - \gw^I_2 \right)\\
=&\ - \square \log \frac{ (\gw^I_2)_+^k \wedge (\gw^I_1)_-^l}{ (\gw^I_1)_+^k \wedge (\gw^I_2)_-^l} - L_X \square u.
\end{align*}
To simplify notation we set $\gw = \gw^I_1,\ \gw_u = \gw^I_1 + \square u$.  Since the vector field $X$ is holomorphic with respect to both $I$ and $J$, it follows that $L_X$ commutes with $I$, $J$, and hence the projection operator $Q$.  It follows that $L_X$ commutes with $\square$, and thus, plugging this fact in above we obtain
\begin{align*}
0 = \square \left( \log \frac{ (\gw^I_2)_+^k \wedge (\gw^I_1)_-^l}{ (\gw^I_1)_+^k \wedge (\gw^I_2)_-^l} + X u \right).
\end{align*}
We note that for functions $f$ satisfying $\square f = 0$ it follows that
\begin{align*}
0 = \left(\tr_{\gw_+} - \tr_{\gw_-} \right) \square f = \tr_{\gw_+} \i \del_+ \delb_+ f + \tr_{\gw_-} \i \del_- \delb_- f = \gD f,
\end{align*}
where $\gD$ denotes the Laplacian with respect to the Chern connection of $\gw$.  By the maximum principle we conclude that the function $f$ must be constant.  Hence there is a constant $c$ such that
\begin{align*}
\log \frac{ (\gw^I_2)_+^k \wedge (\gw^I_1)_-^l}{ (\gw^I_1)_+^k \wedge (\gw^I_2)_-^l} + X u  + c = 0.
\end{align*}
We assume that $c \geq 0$, the case $c \leq 0$ being analogous.  In this case, rearranging the above equation will yield
\begin{align*}
\left(e^{X u} - 1 \right)& \left(\gw_+^k \wedge (\gw_u)_-^l \right)\\
\geq&\ (\gw_u)_+^k \wedge \gw_-^l - \gw_+^k \wedge (\gw_u)_-^l\\
=&\ \i \del_+ \delb_+ u \wedge \Bigl(\sum_{j=0}^{k-1} (\gw_u)_+^j \wedge \gw_+^{k-1-j} \Bigr) \wedge \gw_-^l + \i \del_- \delb_- u \wedge \gw_+^k \wedge \Bigl( \sum_{j=0}^{l-1} (\gw_u)_-^{j} \wedge \gw_-^{l-1-j} \Bigr).
\end{align*}
Dividing this inequality by the background volume form $\gw_+^k \wedge \gw_-^l$ shows that $u$ satisfies
\begin{align*}
\tr_{\til{\gw}} \i \del \delb u + (f X) u \leq 0,
\end{align*}
for some positive definite Hermitian metric $\til{\gw}$ and smooth function $f$.  The strong minimum principle argument will yield that $u$ must be constant.  The case $c \leq 0$ follows in a similar way.
\end{proof}
\end{prop}

To apply Proposition \ref{p:uniqueness} and establish the uniqueness of odd-type generalized K\"ahler-Ricci solitons on Hopf surfaces we first compute the cohomology space $\mathcal H$.
\begin{lemma}\label{lm:cohomology}
	Let $(M^4,g,I_{\alpha\beta},J_{\alpha\beta})$ be an odd-type generalized K\"ahler structure on a diagonal Hopf surface. Then $\dim_\R \mathcal H=2$.
	\begin{proof}
		There is a natural surjective map
		\[
		\mathrm{pr}\colon \mathcal H\to H^{1,1}_{\del+\delb}(M,I)
		:=
		\cfrac
			{\left\{\eta\in \Lambda^{1,1}_{I,\R}\ |\ \del\delb\eta=0 \right\}} 
			{\left\{\delb{\gamma}+\del\bar\gamma\ |\ \gamma\in\Lambda^{1,0}_I\right\}}
		\]
		onto the real (1,1)-Aeppli cohomology of $(M,I)$. This map is well defined, since the equivalence relation of Definition~\ref{d:cone} can be expressed as
		\[
		\left\{
			\delb\gamma+\del\bar{\gamma}\ \big|\ 
			\gamma=\frac{\sqrt{-1}}{2}(\del_+-\del_-)f
		\right\}.
		\]
		The space of real (1,1)-Aeppli cohomology is naturally dual to the space of real (1,1)-Bott-Chern cohomology, and the latter is known to be one-dimensional, see~\cite{Telemancone}. Hence, it remains to prove that $\ker(\mathrm{pr})$ has real dimension 1.
		
		Let $\eta=\delb\gamma+\del\bar{\gamma}$ be a form representing the zero class in $H^{1,1}_{\pd+\bar{\pd}}(M,I)$, where $\gamma=u dz_1+v dz_2$.  The form $\eta$ is formally generalized K\"ahler if and only if $u=u(z_1,z_2,\bar z_1, \bar z_2)$ and $v=v(z_1,z_2,\bar z_1, \bar z_2)$ satisfy
		\begin{equation}\label{eq:cohomology_gk}
		\frac{\del}{\del {z_2}}\bar u=
		\frac{\del}{\del\bar{z_1}} v.
		\end{equation}
		Given $\gamma$ as above, we can define a form
		\[
		\gamma^*=\bar u d\bar{z_1} + vdz_2 \in \Lambda^{0,1}_J
		\]
		which is $\delb_J$-closed if and only if $\gamma$ satisfies~\eqref{eq:cohomology_gk}. In this case $\gamma^*$ defines a class in the Dolbeault cohomology group $H_{\delb}^{0,1}(M,J)\simeq \C$. Now we construct a specific generator of $H_{\delb}^{0,1}(M,J)$.
		
		Let $\Phi(z_1,z_2)$ be a function on $\C^2\backslash\{(0,0)\}$ defined by the identity 
		\[
		|z_1|^2\Phi^{-\frac{2a}{a+b}}+|z_2|^2\Phi^{-\frac{2b}{a+b}}=1.
		\]
		See~\cite[\S 5]{XYang-17} for the proof that this identity yields a well-defined smooth function. Function $\Phi$ is automorphic under the action of the deck transformation, hence $d\log\Phi$ is a well-defined 1-form on $M\simeq S^3\times S^1$ generating $H^1(M,\R)$. This implies that
		\[
		\nu=\delb_J \log \Phi,
		\]
		can be taken as a generator of $H_{\delb}^{0,1}(M,J)$ (\cite[Lemma\,2.3]{Telemancone}). Therefore, the form $\gamma^*$ above can be written as
		\[
		\gamma^*=(c+d\sqrt{-1})\delb_J\log\Phi+\delb_J(\phi+\sqrt{-1}\psi),\quad c,d\in\R;\ \phi,\psi\in C^\infty(M,\R).
		\]
		Directly computing $\eta=\delb\gamma+\del\bar{\gamma}$ in terms of $\gamma^*$, we find:
		\[
		\eta=2d \square \log\Phi+2\square\psi.
		\]
		Thus form $\eta$ represents class $2d[\square\log \Phi]$ in $\mathcal H$. This proves that the kernel of the projection $\mathrm{pr}$ is at most one-dimensional. 
		
		We claim that $[\square\log\Phi]\neq 0$ in $\mathcal H$. Indeed, denote $\omega_I=\sqrt{-1}(\omega_1 dz_1\wedge d\bar{z_1}+\omega_2 dz_2\wedge d\bar{z_2})$ and let $\til\omega_I:=\sqrt{-1}(\omega_1 dz_1\wedge d\bar{z_1}-\omega_2 dz_2\wedge d\bar{z_2})$. Since $\omega_I$ and $\til\omega_I$ are diagonal, and $\omega_I$ is pluriclosed, we conclude that for any $u\in C^\infty(M,\R)$
		\[
		\int_M \til\omega_I\wedge \square u=\int_M\omega_I\wedge \sqrt{-1}\del\delb u=\int_M\delb\del\omega_I\sqrt{-1} u=0.
		\]
		On the other hand, by a direct computation (see~\cite[Lemma\,5.7]{XYang-17}) form $\sqrt{-1}\del\delb \log \Phi$ is nonzero and semipositive. Therefore
		\[
		\int_M \til\omega_I\wedge \square \log\Phi=\int_M\omega_I\wedge \sqrt{-1}\del\delb \log \Phi>0.
		\]
		This proves that $\square\log\Phi$ is a nontrivial element in $\mathcal H$.
	\end{proof}
\end{lemma}
		
	\begin{prop}\label{p:uniqueness_hopf} Let $(M^{4}, g, I_{\alpha\beta}, J_{\alpha\beta})$  be an odd-type generalized K\"ahler-Ricci soliton on a diagonal Hopf surface $M$ with respect to a vector field $X$.
	Then 
	\begin{enumerate}
		\item $g$ is $T^3$-invariant;
		\item after an appropriate rescaling of $(g,X)$:
		\begin{enumerate}
			\item $g = \phi_{t_0}^*g_{\ga\gb}$, where $g_{\ga\gb}$ is the invariant soliton constructed in Proposition~\ref{p:existence}, and $\phi_t\colon M\to M$ is the one-parameter family of diffeomorphisms generated by $Y=\pd_{\Re(u_1)}$;
			\item $X-X_S$ is a Killing vector field, where $X_S=\mu Y$ is the drift vector field of $g_{\ga\gb}$.
		\end{enumerate}
	\end{enumerate}
	\begin{proof}
		Let us first prove that the metric $g$ is $T^3$-invariant. For any $\gamma\in T^3$ consider two generalized K\"ahler solitons:
		\[
		(M,g,I,J),\quad (M, \gamma^*g, I,J).
		\]
		To apply Proposition~\ref{p:uniqueness} we need to check that:
		\begin{itemize}
			\item [(i)] $\gamma_* X=X$;
			\item [(ii)] $[\omega_I]=[\gamma^*\omega_I]$ in $\mathcal H$.
		\end{itemize}
		Using the description of the automorphism groups of primary Hopf surfaces~\cite{NambaHopfAut}, it is easy to see that the automorphism group of $(M,I,J)$ is $G=(\C^*)^2_{z_1,z_2}/\Gamma\simeq T^3\times \R$. Since the vector field $X$ is known to preserve both $I$ and $J$, and $G$ is commutative, we conclude that $\gamma_*X=X$ for any $\gamma\in G$. This proves (i).
		
		To prove (ii) we show that $T^3$ acts trivially on the cohomology space $\mathcal H\simeq \R^2$. Indeed, there is a proper, non-trivial positive cone in $\mathcal H$:
		\[
		\mathcal H_+:=\{[\omega]\in\mathcal H\ |\ \int_{\mathcal{E}_i}\omega>0\},
		\]
		where $\mathcal E_i=\{z_i=0\}, i=1,2$ are the two elliptic curves in $M$. The group $T^3$ acts on $\mathcal H\simeq \R^2$ preserving this positive cone $\mathcal H_+\subset \mathcal H$. Since it is a proper cone in a real two-dimensional space by Lemma \ref{lm:cohomology}, the action must be trivial. Claim (ii) is proved.  Having established (i) and (ii) we are now ready to apply Proposition~\ref{p:uniqueness} to conclude that $\gamma^*g=g$. This proves part (1) of the theorem.
		
		To prove part (2) we observe that $g$, being $T^3$-invariant, (possibly after rescaling by $C>0$) must be of the form~\eqref{f:invmetrics_z}. Furthermore, the vector field $X$ is in the Lie algebra of $G\simeq T^3\times \R$ and must be of the form $X^T+\lambda Y$, where $X^T$ is a Killing vector field in the Lie algebra of $T^3$ and $Y=\pd_{\Re(u_1)}$ is the vector field introduced in the proof of Proposition~\ref{p:existence}. Construction of Proposition~\ref{p:existence} provides an invariant pluriclosed soliton with respect to vector fields $\lambda Y, \lambda\in \R$, and this soliton is unique up to a constant of integration. Different choices of the constant result into shifts in $x=u_1+\bar{u_1}$ coordinate, which is equivalent to pulling back the metric by a one-parameter family of diffeomorphisms $\phi_t^*$ generated by $Y$. Therefore $g=\phi_{t_0}^*g_{\ga\gb}$ for some $t_0$. This proves~(2).
	\end{proof}
	\end{prop}	

\begin{proof}[Proof of Corollary \ref{c:cor1}] Given $(M^4, g, I, J)$ a compact generalized K\"ahler steady soliton, the triple $(M^4, g, I)$ is a pluriclosed soliton, thus it follows from (\cite{Streetssolitons} Theorem 1.1) that either $(M^4, I)$ is K\"ahler Calabi-Yau, or $(M^4, I)$ is biholomorphic to a minimal Hopf surface.  As discussed above, in either the even or odd-type case, the Hopf surface must be diagonal, and the complex structure $J$ is a priori biholomorphic to $I_{\ga \bgb}$, which is also easily checked by the construction itself.  The uniqueness claims in the odd-type case follow from Proposition~\ref{p:uniqueness_hopf}.
\end{proof}

\bibliography{SUCOGKRSCS}
\bibliographystyle{acm}

\end{document}